\documentclass{amsart}
\usepackage{amssymb, latexsym}

\def\sD{{\mathfrak D}}      
   \def\sH{{\mathfrak H}}   
   \def\sK{{\mathfrak K}}   \def\sL{{\mathfrak L}}
\def\sM{{\mathfrak M}}

\def\st{{\mathfrak t}}

\def\s{{\rm s}}

      \def\dC{{\mathbb C}}

\def\bm\chi{\mbox{\boldmath$\chi$}}
\def\half{{\frac{1}{2}}}

\def\ker{{\rm ker\,}}
\def\ran{{\rm ran\,}}
\def\cran{{\rm \overline{ran}\,}}
\def\dom{{\rm dom\,}}
\def\mul{{\rm mul\,}}

\def\cdom{{\rm \overline{dom}\,}}
\def\clos{{\rm clos\,}}
\def\dim{{\rm dim\,}}

\let\xker=\ker \def\ker{{\xker\,}}

\DeclareMathOperator{\hplus}{\, \widehat + \,}
\DeclareMathOperator{\hoplus}{\, \widehat \oplus \,}
\DeclareMathOperator{\hominus}{\, \widehat \ominus \,}

\def\ran{{\rm ran\,}}

\def\dom{{\rm dom\,}}

\def\dim{{\rm dim\,}}
\def\half{{1 \over 2}}

\newtheorem{theorem}{Theorem}[section]
\newtheorem{proposition}[theorem]{Proposition}
\newtheorem{corollary}[theorem]{Corollary}
\newtheorem{lemma}[theorem]{Lemma}
\newtheorem{definition}[theorem]{Definition}

\theoremstyle{definition}

\numberwithin{equation}{section}

\title[Factorization, majorization,  and domination]
{Factorization, majorization, and domination \\ for linear relations}
\author{Seppo Hassi}
\author{Henk~de~Snoo}

\address{Department of Mathematics and Statistics \\
University of Vaasa \\
P.O. Box 700, 65101 Vaasa \\
Finland}
\email{sha@uwasa.fi}

\address{Johann Bernoulli Institute for Mathematics and Computer Science \\
University of Groningen \\
P.O. Box 407, 9700 AK Groningen \\
Nederland}
\email{desnoo@math.rug.nl}

\dedicatory{To our friend Zolt\'an Sebesty\'en on the occasion of his 70th birthday}

\begin{document}

\begin{abstract}
Let $\sH_A$, $\sH_B$, and $\sH$ be Hilbert spaces. Let $A$ be a
linear relation from $\sH$ to $\sH_A$ and  let $B$ be a linear
relation from $\sH$ to $\sH_B$. If there exists an operator $Z \in
\mathbf{B}(\sH_B, \sH_A)$ such that $ZB \subset A$, then $B$ is said
to dominate $A$. This notion plays a major role in the theory of
Lebesgue type decompositions of linear relations and operators.
There is a strong connection to the majorization and factorization
in the well-known lemma of Douglas, when put in the context of
linear relations. In this note some aspects of the lemma of Douglas
are discussed in the context of linear relations and the connections
with the notion of domination will be treated.
\end{abstract}

\maketitle

\section{Introduction}

Let $A$ and $B$ be a pair of linear relations with their domains of
definition in the same Hilbert space $\sH$ and their ranges in the
Hilbert spaces $\sH_A$ and $\sH_B$, respectively. The relation $B$
is said to dominate the relation $A$ if there exists a bounded
linear operator $Z$ from $\sH_B$ to $\sH_A$ such that $ZB \subset
A$. Domination is preserved when the closures of $A$ and $B$ are
considered. In the particular case that $A$ and $B$ are, not
necessarily densily defined, operators this is equivalent to $\dom A
\subset \dom B$ and the existence of a constant $c\ge 0$ such that
$\|Af\| \le c \|Bf\|$ holds for all $f \in \dom A$. The notion of
domination, which is familiar from measure theory, plays an
important role in the theory of Lebesgue type decompositions. This
notion and its role in Lebesgue type decompositions for a pair of
bounded operators go back to Ando \cite{Ando}; it has a similar
position when decomposing a nonnegative form with respect to an
another nonnegative form, see \cite{HSS1}, or when decomposing an
unbounded operator or a linear relation \cite{HSS2,HSSSz2007,HSSz0},
where some further history and references can be found.

In the present paper it will be shown that domination is closely related to
the following well-known lemma of R.G. Douglas \cite{D} when that
lemma is put in the context of unbounded linear operators or,
more generally,  linear relations.

\begin{lemma}[Douglas]\label{Douglas}
Let $A, B \in \mathbf{B}(\sH,\sK)$,  the bounded 
everywhere defined linear operators from
a Hilbert space $\sH$ to a Hilbert space $\sK$. Then the following
statements are equivalent:
\begin{enumerate}\def\labelenumi{\rm (\roman{enumi})}
\item $\ran A \subset \ran B$;
\item $A=BW$ for some bounded linear operator $W \in \mathbf{B}(\sH)$;
\item $AA^{*} \le \lambda BB^{*}$ for some $\lambda \ge 0$.
\end{enumerate}
\end{lemma}

If the equivalent conditions (i) -- (iii) hold,
then there is a unique operator $W$ such that
\begin{enumerate}\def\labelenumi{\rm (\alph{enumi})}
\item $\|W\|^2 =\inf\{\mu :\,AA^*\leq\mu BB^*\}$;
\item $\ker A = \ker W$;
\item $\ran W \subset \cran B^*$.
\end{enumerate}
In the literature one can find a statement which is equivalent to
the three items in Lemma~\ref{Douglas}, namely
\begin{enumerate}\def\labelenumi{\rm (\roman{enumi})}\setcounter{enumi}{3}
\item $AA^{*}=BMB^{*}$, where $M \in \mathbf{B}(\sH)$ is nonnegative
and $\|M\| \le \lambda$.
\end{enumerate}
One may take $\ran M \subset \cran B^{*}$.
In addition to the results in the above lemma
Douglas indicated some further results for the
case when $A$ and $B$ are densely defined closed linear operators;
see \cite{D}.
Various extensions of these basic results by Douglas can be found in
the literature; see, for instance, \cite{BB,E,F}. The factorization
aspect of the Douglas lemma was recently put in the context of
linear relations by D. Popovici and Z. Sebesty\'en \cite{PS}; see also
some refinements by A. Sandovici and Z. Sebesty\'en \cite{SS}.
For the majorization aspect of the Douglas lemma, see \cite{AH}.

The contents of the present paper are now briefly explained.
For closed linear operators or relations $A$ and $B$  
the following equivalence will be established in Theorem \ref{een}:
\[
 A \subset BW \quad \Leftrightarrow \quad AA^{*} \le c^{2} BB^{*},  
\]
where $W$ is a bounded linear operator and $c \ge 0$, in fact $\|W\| \le c$.  
This result characterizes  majorization in terms of a simple factorization type inclusion.
Domination for a pair of closed 
linear operators or relations  can be characterized in a similar way:  
\[
ZB \subset A \quad \Leftrightarrow \quad A^{*}A \le c^{2} B^{*}B,
\] 
see Theorem \ref{een+}. Some consequences of these results
 will be explored in Section~3 and Section~4.
In particular, a characterization of the equalities $A=BW$ and $ZB=A$ 
is given. For bounded linear operators   the factorization 
$A=BW$ in the original Douglas lemma can be directly connected to 
the notion of domination for linear relations by means of the following observation:
\[
A=BW \quad \Leftrightarrow \quad WA^{-1} \subset B^{-1},
\]
see Lemma \ref{douglasfact}.
This last equivalence, when combined with the two earlier equivalences, 
provides a simple proof for the characterization of the ordering of 
nonnegative selfadjoint relations in terms of resolvents; see Theorem \ref{twee}.
For the convenience of the reader some results concerning
closed nonnegative forms and associated linear relations will be
recalled in Section 2.

\section{preliminaries}

Let  $H$ be a linear relation from a Hilbert space $\sH$ to a
Hilbert space $\sK$; i.e., $H$ is a linear subspace of the product
$\sH \times \sK$. The domain, range, kernel, and multivalued part of
$H$ are denoted by $\dom H$, $\ran H$, $\ker H$, and $\mul H$. The
formal inverse $H^{-1}$ of $H$ is a relation from $\sK$ to $\sH$,
defined by $H^{-1}=\{ \{f',f\}:\, \{f,f'\} \in H\}$, so that $\dom
H^{-1}=\ran H$, $\ran H^{-1}=\dom H$, $\ker H^{-1}=\mul H$, and
$\mul H^{-1}=\ker H$. For $\sL\subset\sH$ the set $H(\sL)$ is a
subset of $\sK$ defined by
\[
 H(\sL)=\{\,h':\, \{h,h'\}\in H \text{ for some } h\in\sL\,\}.
\]
In particular, $H(\{0\})=\mul H$.

Let $H_1$ and $H_2$ be relations from a Hilbert space $\sH$ to a
Hilbert space $\sK$. Then $H_1$ is a  \textit{restriction} of $H_2$
and $H_2$ is an \textit{extension} of $H_1$ if $H_1 \subset H_2$.

\begin{proposition}\label{arens0}
Let $H_1$ and $H_2$ be relations from a Hilbert space $\sH$ to a
Hilbert space $\sK$ and assume that $H_1 \subset H_2$. Then the
following statements are equivalent:
\begin{enumerate}\def\labelenumi{\rm (\roman{enumi})}

\item $\dom H_1=\dom H_2$;

\item $H_2=H_1 \hplus ( \{0\} \times \mul H_2)$.

\end{enumerate}
Moreover, the following statements are equivalent:
\begin{enumerate}\def\labelenumi{\rm (\roman{enumi})}
\setcounter{enumi}{2}

\item $\ran H_1 =\ran H_2$;

\item $H_2=H_1 \hplus ( \ker H_2 \times \{0\})$.

\end{enumerate}
\end{proposition}

\begin{proof}
By symmetry it suffices to show the equivalence between (i) and
(ii).

(i) $\Rightarrow$ (ii) It suffices to show that $H_2 \subset H_1
\hplus (\{0\} \times \mul H_2)$. Let $\{h,h'\} \in H_2$.  Since $h
\in \dom H_2 \subset \dom H_1$, there exists an element $k' \in \sK$
such that $\{h,k'\} \in H_1$. Hence, with $\varphi'=h'-k'$, it
follows that
\[
 \{h,h'\} = \{h,k'\}+\{0,\varphi'\},
\]
and thus $\{0,\varphi'\} \in H_2$ or $\varphi' \in \mul H_2$. Hence
(ii) follows.

(ii) $\Rightarrow$ (i) This implication is trivial.
\end{proof}

The useful result in the following corollary can be found in
\cite{Arens}.

\begin{corollary}\label{arens1}
Let $H_1$ and $H_2$ be relations from a Hilbert space $\sH$ to a
Hilbert space $\sK$ and assume that $H_1 \subset H_2$. Then the
following statements are equivalent:
\begin{enumerate}\def\labelenumi{\rm (\roman{enumi})}

\item $H_1=H_2$;

\item $\dom H_1 = \dom H_2$ and $\mul H_1 = \mul H_2$;

\item $\ran H_1 = \ran H_2$ and $\ker H_1 = \ker H_2$.

\end{enumerate}
\end{corollary}

 \begin{corollary}\label{arens2}
Let $H_1$ and $H_2$ be relations from a Hilbert space $\sH$ to a
Hilbert space $\sK$ and assume that $H_1 \subset H_2$. Then
\begin{enumerate}\def\labelenumi{\rm (\roman{enumi})}
\item $\dom  H_1=\sH$,  $\mul  H_2=\{0\}$ $\Rightarrow$ $H_1=H_2$;
\item $\ran H_1=\sK$, $\ker H_2=\{0\}$ $\Rightarrow$ $H_1=H_2$.
\end{enumerate}
\end{corollary}

The sum of two linear relations $H_1$ and $H_2$ from $\sH$ to $\sK$
is a linear relation defined by
\[
 H_1+H_2=\{\{f, f'+f''\}:\, \{f,f'\} \in H_1, \,\{f,f''\} \in H_2\},
\]
while their componentwise sum is a linear relation defined by
\[
 H_1 \hplus H_2=\{\{f+g, f'+g'\}:\, \{f,f'\} \in H_1, \,\{g,g'\} \in H_2\}.
\]
Let $H_1$ be a relation from a Hilbert space $\sH$ to a Hilbert space $\sM$
and let  $H_2$ be a relation from a Hilbert space $\sM$ to a Hilbert space $\sK$.
The product $H_2H_1$ is a linear relation from $\sH$ to $\sK$ defined by
\begin{equation}\label{prod}
 H_2 H_1=\{\{f, f'\}:\, \{f,\varphi\} \in H_1, \,\{\varphi,f'\} \in H_2 \,\,
  \mbox{for some}\,\, \varphi \in \sM\}.
\end{equation}
Observe, that
\begin{equation}\label{kerAB}
 \ker (H_2H_1)= H_1^{-1}(\ker H_2)
  =\{ f\in\sH:\, \{f,\varphi\}\in H_1\,\, \mbox{for some}\,\, \varphi \in \ker H_2\},
\end{equation}
and
\begin{equation}\label{mulAB}
 \mul H_2H_1=H_2(\mul H_1)=\{f'\in\sK:\,\{\varphi,f'\}\in H_2\,\,
 \mbox{for some}\,\, \varphi \in \mul H_1\}.
\end{equation}
In particular, $\ker H_1\subset \ker H_1H_2$ and $\mul H_2\subset
\mul H_2H_1$. The following identities are also easy to check:
\begin{equation}\label{BBinv}
 HH^{-1}=I_{\ran H}\,{\widehat +}\,\bigl(\{0\}\times\mul H\bigr)
 \quad\text{and}\quad
 H^{-1} H=I_{\dom H}\,{\widehat+}\,\bigl(\{0\}\times\ker H\bigr)
\end{equation}
with both sums direct. Hence, in particular,
\begin{equation}\label{BBinv+}
 \mul H=\{0\} \,\, \Rightarrow \,\, HH^{-1}=I_{\ran H}; \quad  
\ker H=\{0\} \,\, \Rightarrow \,\, H^{-1} H=I_{\dom H}.
\end{equation}

The closure of a linear relation $H$ from $\sH$ to $\sK$
is the closure of the linear subspace in $\sH \times \sK$,
when the product is provided with the product topology.
The closure of an operator need not be an operator;
if it is then one speaks of a closable operator.
The relation $H$ is called closed when it is closed
as a subspace of $\sH \times \sK$.
In this case both $\ker H \subset \sH$ and
$\mul H \subset \sK$ are closed subspaces.

Let $H$ be a closed linear relation from $\sH$ to $\sK$.
Then $H_{\rm mul}=\{0\} \times \mul H$ is a closed linear relation
and $H_{\rm s}=H \hominus H_{\rm mul}$, so that
$\dom H_{\rm s}=\dom H$ is dense in
$\cdom H=\sH\ominus \mul H^*$, while
$\ran H_{\rm s} \subset \cdom H^*= \sK \ominus \mul H$.
The operator part $H_{\rm s}$ and $H_{\rm mul}$ lead
to the componentwise orthogonal decomposition
\begin{equation}\label{operpart}
 H=H_{\rm s} \hoplus H_{\rm mul}.
\end{equation}
The adjoint relation $H^*$ from $\sK$ to $\sH$ is
defined by $H^*=JH^\perp=(JH)^\perp$, where
$J\{f,f'\}=\{f',-f\}$. The adjoint is automatically
a closed linear relation and the closure of $H$ is
given by $H^{**}$. The operator part $(H^{*})_{\s}$
 is densely defined in $\cdom H^{*}=\sH \ominus \mul H^{**}$
and maps into  $\cdom H=\cdom H^{**}=\sH \ominus \mul H^*$.
When $H$ is closed the operator
parts $H_{\rm s}$ and $(H^{*})_{\s}$ are connected by
\begin{equation}\label{Xadj}
(H_{\rm s})^{\times}=(H^{*})_{\s},
\end{equation}
where $(H_{\s})^{\times}$ denotes the adjoint of $H_{\s}$ in the
sense of the smaller spaces $\cdom H$ and $\cdom H^*$.

Let $H_1$ be a relation from a Hilbert space $\sH$
to a Hilbert space $\sM$
and let  $H_2$ be a relation from a Hilbert space $\sM$
to a Hilbert space $\sK$.
The product satisfies
\begin{equation}\label{Bbadj0}
 H_1^*H_2^* \subset (H_2H_1)^*.
\end{equation}
Moreover, if $H_2 \in \mathbf{B}(\sM,\sK)$ then there is actually
equality
\begin{equation}\label{Bbadj}
 H_1^*H_2^* = (H_2H_1)^*,
\end{equation}
see \cite[Lemma~2.4]{HSS2005}, so that, in particular
\[
H_2 H_1^{**} \subset (H_2H_1)^{**}.
\]
Assume that the relations $H_{1}$ and $H_{2}$ are closed.
In general the product $H_{2}H_{1}$ is not closed.
However, if for instance  $H_{1} \in \mathbf{B}(\sH,\sM)$,
then the product $H_{2}H_{1}$ is closed.

A linear relation $H$ in a Hilbert space $\sH$
is symmetric if $H \subset H^{*}$
and selfadjoint if $H = H^{*}$.  If the relation $H$ is selfadjoint
then $H_\s$ is a selfadjoint
operator in $\cdom H=\sH \ominus \mul H$.
A linear relation $H$ in a Hilbert space $\sH$ is nonnegative if $(h',h) \ge 0$
for all $\{h,h'\} \in H$. In particular a nonnegative relation is symmetric.

An important special case of a nonnegative selfadjoint relation
appears when one considers relations of the form $T^{*}T$ where $T$
is a closed linear relation from a Hilbert space $\sH$ to a Hilbert
space $\sK$; cf. \cite{SeTa2012}.

\begin{lemma}\label{T*T}
Let $T$ be a closed  relation from a Hilbert space $\sH$ to a
Hilbert space $\sK$.  Then the product $T^{*}T$ is a nonnegative
selfadjoint relation in $\sH$. Furthermore,
\begin{equation}\label{henil0}
 T^*T=T^*T_{\rm s} = (T_{\rm s})^*T_{\rm s},
\end{equation}
so that in particular
\begin{equation}\label{heni}
\ker (T^*T)=\ker T=\ker T_{\rm s}, \quad
 \mul (T^*T)=\mul T^*=\mul (T_{\rm s})^*.
\end{equation}
The operator part of $T^* T$ can be rewritten as
\begin{equation}\label{s-sing+}
 (T^* T)_{\rm s}= (T^*)_{\rm s}T_{\rm s}=(T_{\rm s})^{\times}T_{\rm s}.
\end{equation}
\end{lemma}

\begin{proof}
It is clear from the definition  that $T^*T$ is a nonnegative
selfadjoint relation in $\sH$. In fact $T^*T$ is selfadjoint since
$\ran (T^*T+I)=\sH$, which follows from $\sH^2=T \hoplus T^\perp=T
\hoplus JT^*$.

Next, let $P$ be the orthogonal projection from $\sK$ onto $\cdom
T^*$, so that $PT=T_{\rm s}\subset T$. Since $\dom T^{*} \subset
\cdom T^{*}=\ran P$ it also follows that $T^* \subset T^*P$. Hence
\begin{equation}\label{pre}
 T^*T \subset T^*P T   \subset T^* T,
\end{equation}
so that the inclusions are both equalities. From $PT=T_{\rm s}$ one
obtains that $(T_{\rm s})^*=(PT)^*=T^*P$, so that \eqref{pre} leads
to \eqref{henil0}. Since $T_{\rm s}$ is an operator, \eqref{heni} is
immediate from \eqref{henil0}. Furthermore, \eqref{s-sing+} is clear
from \eqref{Xadj}.
\end{proof}

\begin{lemma}\label{sqrt}
Let $H$ be a nonnegative selfadjoint relation in a Hilbert space $\sH$.
Then there exists a unique nonnegative selfadjoint relation $K$ in $\sH$,
denoted by  $K=H^{\half}$, such that $K^{2}=H$.
Moreover, $H^{\half}$ has the representation
\begin{equation}\label{wortel}
H^{\half}=H_{\s}^{\half} \hoplus H_{\mul}.
\end{equation}
\end{lemma}

\begin{proof}
It is clear that $K$ defined by the right hand side of \eqref{wortel}
is a nonnegative selfadjoint relation with $\mul K=\mul H$.
To see that $K^{2}=H$, let $\{f,f'\} \in K^{2}$. Then  $\{f,\varphi\} \in K$ and
$\{\varphi,f'\} \in K$. Clearly,
\[
\varphi=H_{s}^{\half}f+\alpha, \quad f'=H_{s}^{\half} \varphi +\beta,
\]
with $\alpha \in \mul H$ and $\beta \in \mul H$.
Since $\varphi \in \dom H_{s}^{\half}$
it follows that $\alpha=0$ and $f'=H_{s}f+\beta$,
so that $\{f,f'\} \in H$. It follows that
$K^{2} \subset H$, and since $K^{2}=K^{*}K$ is selfadjoint,
it follows that $K^{2}=H$.

In order to show uniqueness, let $K$ be a nonnegative
selfadjoint relation such that $K^{2}=H$.
Then
\[
\mul K=\mul H.
\]
To see this let $\{0,\psi\} \in K$,
then clearly $\{0,\psi\} \in K^{2}=H$,
so that $\mul K \subset \mul H$.
For the reverse inclusion, let $\{0,\psi\}\in
H=K^{2}$. Then $\{0, \varphi\} \in K$ and $\{\varphi, \psi \} \in K$.
Since $K$ is selfadjoint
it follows that $\varphi=0$, so that $\{0,\psi \} \in K$
and $\mul H \subset \mul K$.
This implies that $K=K_{s} \oplus H_{\mul}$,
where $K_{s}$ is a nonnegative selfadjoint operator.
It will now be shown that $H_{s}=(K_{\s})^{2}$,
and since the square root of
a nonnegative selfadjoint operator is uniquely
determined it follows that
$K_{\s}=H_{\s}^{1/2}$.

To see that $H_{\s}=K_{\s}^{2}$, let $\{f,f'\} \in H_{\s}$.
Then $\{f,f'\} \in H=K^{2}$ and
$f' \perp \mul H$. Now $\{f,\varphi\} \in K$ and
$\{\varphi, f'\} \in K$ for some
$\varphi \in \cdom K=\cdom H$. Hence $\{f, \varphi\} \in K_{\s}$
and $\{\varphi, f'\} \in K_{\s}$,
so that $\{f,f'\} \in K_{\s}^{2}$. Thus $H_{\s} \subset (K_{\s})^{2}$.
For the converse inclusion, let $\{f,f'\} \in (K_{\s})^{2}$.
Then $\{f,\varphi\} \in K_{\s}\subset K$,
$\{\varphi,f'\} \in K_{\s} \subset K$,
so that $\{f,f'\} \in K^{2}=H$.  Since $f' \perp \mul H$,
it follows that $\{f,f'\} \in H_{\s}$.
\end{proof}

Let $H$ be a nonnegative selfadjoint relation.
Since Lemma \ref{sqrt} implies that
$\mul H^\half=\mul H$, it follows that
\[
(H^\half)_{\rm s}=(H_{\rm s})^\half,
\]
so that the notation $H^\half_\s$ is unambiguous.
Furthermore it is clear that
\begin{equation}\label{incll}
\dom H \subset \dom H^\half \subset \cdom H^\half=\cdom H.
\end{equation}
Therefore the following statements are equivalent:
\begin{equation}\label{DD}
 \dom H \text{\,\,closed};  \quad \dom H^\half \text{\,\,closed};
 \quad \dom H =\dom H^\half.
\end{equation}
Let $H$ be a nonnegative selfadjoint relation.
Then for each $x>0$,
\begin{equation}\label{hs1}
 \dom(H + x)^{1/2}= \dom H^{1/2},
\end{equation}
and, moreover,
\begin{equation}\label{hs2}
 \| (H_{s} + x)^{1/2} h\|^{2} = \| (H^{1/2})_s h\|^{2} +x \|
h\|^{2} , \quad h \in \dom H^{1/2}.
\end{equation}
It is clear that the identity holds for $h \in \dom H$ and
since $\dom H$ is a core for $H^{1/2}$ it holds for $h \in \dom H^{1/2}$.
  
There is a natural ordering for nonnegative selfadjoint relations
in a Hilbert space $\sH$; it is inspired by the corresponding situation for
selfadjoint operators $H_1, H_1 \in \mathbf{B}(\sH)$.
Two nonnegative selfadjoint relations $H_1$ and $H_2$
are said to satisfy the inequality $H_1 \leq H_2$ if
\begin{equation}\label{ineq1}
\dom H_{2\s}^{\half} \subset \dom H_{1\s}^{\half}, \quad
\| H_{1\s}^{\half} h \| \leq \| H_{2\s}^{\half} h \|, \quad
h \in \dom H_2^{\half}.
\end{equation}
It follows from \eqref{hs1} and \eqref{hs2}  
that $H_{1} \leq H_{2}$ if and only if $H_1+x \leq H_2+x$ 
for some (and hence for all) $x >0$.
 
A sesquilinear form (or form for short) $\st[\cdot,\cdot]$
in a Hilbert space $\sH$
is a mapping from $\sD \times \sD$ to $\dC$
where $\sD$ is a (not necessarily densely defined)
linear subspace of $\sH$, such that
it is linear in the first  entry and anti-linear in the second entry.
The domain $\dom \st$ is defined by $\dom \st=\sD$.
The corresponding quadratic form $\st[\cdot]$ is defined by
$\st[\varphi]=\st[\varphi,\varphi]$, $\varphi \in \dom \st$.
A  sesquilinear form $\st$ is said to be nonnegative if
\[
\st[\varphi] \ge 0, \quad \varphi \in \dom \st.
\]
The semibounded form $\st$ in $\sH$ is said to be \emph{closed}
if for any sequence $(\varphi_n)$
in $\dom \st$ one has
\begin{equation}
 \varphi_{n} \to \varphi, \quad
\st[\varphi_{n} - \varphi_{m}] \to 0,
 \quad \Rightarrow \quad
 \varphi \in \dom \st,\quad \st[\varphi_{n} - \varphi] \to 0.
\end{equation}
The inequality $\st_1 \leq \st_2$ for forms $\st_1$ and
$\st_2$ is defined by
\begin{equation}\label{ineq0}
 \dom \st_2 \subset \dom \st_1,
 \quad
 \st_1[h] \leq \st_2[h],
 \quad
 h \in \dom \st_2.
\end{equation}
In particular, $\st_2 \subset \st_1$ implies $\st_1 \leq \st_2$.

The theory of nonnegative forms can be found in  \cite{Kato}. The
representation theorem gives a connection between nonnegative
selfadjoint relations and nonnegative closed forms; see
\cite{HSSW,Kato}.

\begin{theorem}[representation theorem]\label{repre}
Let $\st$  be a closed nonnegative form in the Hilbert space $\sH$.
Then there exists a nonnegative selfadjoint relation $H$
in $\sH$ such that
\begin{enumerate}\def\labelenumi{\rm (\roman{enumi})}
\item $\dom H \subset \dom \st$ and
\begin{equation}\label{eqn:firstrepresentation}
 \st[\varphi,\psi] = ( \varphi',\psi )
\end{equation}
 for every $\{ \varphi,\varphi' \} \in H$ and $\psi \in \dom \st$;
\item $\dom H$ is a core for $\st$ and $\mul H=(\dom \st)^{\perp}$;
\item if $\varphi \in \dom \st$, $\omega \in \sH$, and
\begin{equation}\label{eqn:vergelijking_1st_representation(3)}
 \st[\varphi,\psi] = ( \omega,\psi )
\end{equation}
holds for every $\psi$ in a core of $\st$,
then  $\{ \varphi,\omega \} \in H$.
\end{enumerate}
The nonnegative selfadjoint relation $H$ is uniquely determined by
(i).
\end{theorem}

The following result is a direct consequence of the
representation theorem.

\begin{proposition}\label{cor+}
Let $T$ be a closed linear relation from a Hilbert space
$\sH$ to a Hilbert space $\sK$.
The nonnegative selfadjoint relation $T^*T$ in the Hilbert space $\sH$
corresponds to the closed nonnegative form
\begin{equation}\label{henil}
 \st[h,k]=(T_{\rm s} h,T_{\rm s} k)_\sK,
 \quad h,k \in \dom \st=\dom T_{\rm s} = \dom T,
\end{equation}
and, in particular,
\begin{equation}\label{hhhenil}
 \st[h,k]=( (T_{\rm s})^{\times} T_{\rm s} h, k)_\sK,
 \quad h \in \dom T^*T, \quad k \in \dom \st.
\end{equation}
\end{proposition}

\begin{proof}
Since $T_\s$ is a closed linear operator, it follows that the form
in \eqref{henil} is closed. Clearly, if in \eqref{henil} one assumes
that $h \in \dom T^*T=\dom (T_{\rm s})^{\times} T_{\rm s}$, see
\eqref{s-sing+}, then \eqref{hhhenil} follows. The result is now
obtained from Theorem \ref{repre}.
\end{proof}

Proposition~\ref{cor+} combined with \eqref{s-sing+} in
Lemma~\ref{T*T} yields the so-called second representation theorem
for closed forms.

\begin{corollary}
Let $\st$  be a closed nonnegative form in the Hilbert space $\sH$
and let $H$ be the corresponding nonnegative selfadjoint relation $H$ in $\sH$
as in Theorem \ref{repre}. Then
\begin{equation}
 \dom \st = \dom H_{\s}^{\half}
 \quad \mbox{and} \quad
 \st[ \varphi,\psi ] = (H_{\s}^{\frac{1}{2}} \varphi,H_{\s}^{\frac{1}{2}} \psi),
 \quad \varphi,\psi \in \dom \st.
\end{equation}
A subset of $\dom \st = \dom H_{\s}^{\frac{1}{2}}$ is a core of the
form $\st$ if and only if it is a core of the operator
$H_{\s}^{\frac{1}{2}}$. In particular,  $\dom H$ is a core of
$H^{\frac{1}{2}}$.
\end{corollary}

As a straightforward consequence of the representation theorem
one can state the following result which connects inequalities
between nonnegative selfadjoint relations with inequalities
between the corresponding nonnegative closed forms.

\begin{corollary}\label{cor}
Let $\st_1$ and $\st_2$ be closed nonnegative forms
and let $H_1$ and $H_2$ be the corresponding nonnegative
selfadjoint relations. Then
\begin{equation}\label{mainineq}
\st_1 \leq \st_2 \quad {\text{if and only if}}
\quad H_1 \leq H_2.
\end{equation}
\end{corollary}

\begin{corollary}\label{repr+}
Let $\sH$, $\sH_1$, and $\sH_2$ be Hilbert spaces.
Let $T_1$   be a closed linear relation from
$\sH$ into $\sK_1$
and let $T_2$   be a closed linear relation from
$\sH$ into $\sK_2$.
Then $T_1^*T_1 \leq T_2^*T_2$ if and only if
\[
    \dom T_2 \subset \dom T_1 \quad \mbox{and} \quad
    \|(T_1)_\s h\|_{\sK_1} \leq \|(T_2)_\s h\|_{\sK_2} , \quad h \in \dom T_2.
\]
\end{corollary}

\begin{proof}
Let $\st_1$ and $\st_2$ be the closed nonnegative forms in the Hilbert space $\sH$
induced by $T_1^*T_1$ and $T_2^*T_2$.
Hence by Corollary \ref{cor} one has $T_1^*T_1 \leq T_2^*T_2$ if and only if
$\st_1 \leq \st_2$. By Proposition \ref{cor+} $\st_1 \leq \st_2$ if and only if
\[
\dom T_2 \subset \dom T_1, \quad 
((T_{1})_\s h,(T_{1})_\s h)_{\sK_1}  \leq ((T_{2})_\s h,(T_{2})_\s h)_{\sK_2} ,
\quad  h \in \dom T_2. \qedhere
\]
\end{proof}

\section{The lemma of Douglas in the context of linear relations}

In this section the lemma of Douglas, see Introduction, will be
discussed in the context of linear relations. The first result
to be presented is about range inclusion and factorisation. It goes
back to D. Popovici and Z. Sebesty\'en \cite{PS}, who stated it
actually in the context of linear spaces. Some refinements can be
found in \cite{SS}.

\begin{proposition}\label{thm1}
Let $\sH_A$, $\sH_B$, and $\sH$ be Hilbert spaces, let $A$ be a
linear relation from $\sH_A$ to $\sH$, and let $B$ be a linear
relation from $\sH_B$ to $\sH$. Then $\ran A \subset \ran B$ if and
only if there exists a linear relation $W$ from $\sH_A$ to $\sH_B$
such that $A \subset BW$.
\end{proposition}

\begin{proof}
($\Rightarrow$) Let the linear relation $W$ from $\sH_A$
to $\sH_B$ be defined by the product
\[
W=B^{-1}A.
\]
Let $\{f,f'\} \in A$. Then $f' \in \ran A$, so that $f' \in \ran B$
and there exists
$\varphi \in \sH_B$ such that $\{\varphi, f'\} \in B$
or $\{f', \varphi\} \in B^{-1}$.
Hence $\{f,\varphi\} \in W$ and $\{f,f'\} \in BW$.

($\Leftarrow$) Let $f' \in \ran A$, then for some $f \in \sH_A$
one has $\{f,f'\} \in A$.
Hence there is $\varphi \in \sH_B$ such that $\{f, \varphi\} \in W$
and $\{\varphi,f'\} \in B$.
This implies that $f' \in \ran B$.
\end{proof}

For the next result, see \cite[Proposition~2]{SS}; for
completeness a short proof is presented.

\begin{proposition}
Let $\sH_A$, $\sH_B$, and $\sH$ be Hilbert spaces, let $A$ be a
linear relation from $\sH_A$ to $\sH$, and let $B$ be a linear
relation from $\sH_B$ to $\sH$. Then there exists a linear relation
$W$ from $\sH_A$ to $\sH_B$ such that $A = BW$ if and only if
\[
 \ran A \subset \ran B \quad \text{and}\quad \mul B \subset \mul A.
\]
\end{proposition}

\begin{proof}
($\Rightarrow$) It follows from \eqref{mulAB} that $\mul B\subset
\mul BW= \mul A$ while $\ran A \subset \ran B$ holds by 
Proposition \ref{thm1}.

($\Leftarrow$) As in the proof of Proposition \ref{thm1} consider
$W=B^{-1}A$ which satisfies $A \subset BW$. In view of \eqref{BBinv}
one can write
\begin{equation}\label{BW2}
 BW=BB^{-1}A=\left(I_{\ran B} \hplus (\{0\}\times \mul B) \right)A.
\end{equation}
Since $\ran A\subset \ran B$, it is clear from \eqref{BW2} that
$\dom BW=\dom A$ and since $\mul B \subset \mul A$ one also
concludes from \eqref{BW2} that $\mul BW=\mul A$. Therefore, the
equality $BW=A$ holds by Corollary \ref{arens1}.
\end{proof}

Observe that if $W$ is a linear relation from $\sH_A$ to $\sH_B$, 
then the inclusion $A \subset BW$ shows that 
\[
\dom A \subset \dom W \quad \mbox{and} \quad \ran A \subset \ran B.
\] 
Furthermore, if  $W$ is an operator, then the inclusion
$A \subset BW$ is equivalent to:
\[
\dom A\subset \dom W \quad \mbox{and} \quad
 \{Wf,f'\} \in B \quad \mbox{for all} \quad \{f,f'\} \in A,
\]
so that in particular $W$ takes $\dom A$ into $\dom B$. Hence
when the relation $W$ is a bounded operator then it may be assumed 
that $W \in \mathbf{B}(\cdom A,\cdom B)$. 
In this case the zero continuation $W_c$ of $W$ to
$(\dom A)^\perp$ satisfies $A \subset BW\subset BW_c$ and
$\|W_c\|=\|W\|$, so that without loss of generality it may be assumed that
$W \in \mathbf{B}(\sH_{A},\sH_{B})$. 

\begin{lemma}\label{lemma1}
Let $A \subset BW$ for some $W \in \mathbf{B}(\sH_A,\sH_B)$.
Then
\[
W^{*}B^{*}  \subset A^{*} \,\text{ and }\, A^{**} \subset B^{**}W.
\]
\end{lemma}

\begin{proof}
Clearly $A \subset BW$ implies via \eqref{Bbadj0} that
\[
   W^{*}B^{*} \subset  (BW)^{*} \subset A^{*}.
\]
This inclusion combined with $W^* \in \mathbf{B}(\sH_B, \sH_A)$ and
\eqref{Bbadj} in turn gives rise to
\[
 A^{**} \subset (W^{*}B^{*})^{*}=B^{**}W^{**}=B^{**}W. \qedhere
\]
\end{proof}

The main result in this section concerns factorization and
majorization. If $A$ and $B$ are closed linear relations, then the
case that $W \in \mathbf{B}(\sH_A,\sH_B)$ can be characterized as
follows; see also \cite{AH,D}.

\begin{theorem}\label{een}
Let $\sH_A$, $\sH_B$, and $\sH$ be Hilbert spaces, let $A$ be a
closed linear relation from $\sH_A$ to $\sH$, and let $B$ be a
closed linear relation from $\sH_B$ to $\sH$. Then there exists an
operator $W \in \mathbf{B}(\sH_A,\sH_B)$ {{\rm (}}or equivalently an
operator $W \in \mathbf{B}(\cdom A,\cdom B)${\rm{)}} such that
\begin{equation}\label{ABC}
 A \subset BW,
\end{equation}
if and only if there exists $c \ge 0$ such that
\begin{equation}\label{AB}
 AA^{*} \le c^{2} BB^{*}.
\end{equation}
One can take $\|W\| \le c$.
\end{theorem}

\begin{proof}
($\Rightarrow$) Let $A \subset BW$ with 
$W \in \mathbf{B}(\cdom A,\cdom B)$.
By considering the zero continuation of $W$, 
again denoted by $W$, it may be assumed
that $W \in \mathbf{B}(\sH_A,\sH_B)$. Then
\begin{equation}\label{ABC1}
 W^{*}B^{*} \subset A^{*},
\end{equation}
cf. Lemma \ref{lemma1}. In particular this implies that $\dom B^{*}
\subset \dom A^{*}$. Now let $\{f,f'\} \in (B^{*})_{\s} \subset
B^{*}$. Then it follows from \eqref{ABC1} that
\[
 \{f,W^{*}f'\} \in A^{*}.
\]
Hence there is an element $\chi \in \mul A^{*}$ such that
\[
 W^{*}(B^{*})_{\s}f=(A^{*})_{\s}f +\chi.
\]
Observe that
\[
\|(A^{*})_{\s}f\|^{2} \le  \|(A^{*})_{\s}f\|^{2} +\|\chi\|^{2}
=\|W^{*}(B^{*})_{\s}f\|^{2}
\le \|W\|^2 \|(B^{*})_{\s}f\|^{2}.
\]
Together with $\dom B^{*} \subset \dom A^{*}$ this inequality
proves \eqref{AB};
see Corollary \ref{repr+}.
 
($\Leftarrow$)  Assume that \eqref{AB} holds, in other words,
assume that there exists $c \ge 0$ such that
\begin{equation}\label{AB+}
\dom B^{*} \subset \dom A^{*}, \quad  c\, \|(B^{*})_{\s}f\| \ge
\|(A^{*})_{\s}f\|, \quad f \in \dom B^{*}.
\end{equation}
Consider $A_{\s}$ as a densely defined operator from
$\cdom A$ to $(\mul A)^{\perp}$
and $B_{\s}$ as a densely defined operator from
$\cdom B$ to $(\mul B)^{\perp}$.
Then the assumption \eqref{AB+} is equivalent to
\begin{equation}\label{AB2}
\dom B^{*} \subset \dom A^{*}, \quad c \,\|(B_{\s})^{\times}f\| \ge
\|(A_{\s})^{\times}f\|, \quad f \in \dom B^{*},
\end{equation}
where the adjoints $(A_{\s})^{\times}$ and $(B_{\s})^{\times}$ are
with respect to these smaller spaces; see \eqref{Xadj}. Define the
linear relation $D$ by
\[
 D=\{\{(B_{s})^{\times}f, (A_{\s})^{\times}f\} :\, f \in \dom B^{*} \}.
\]
Then by \eqref{AB2} $D$ is a bounded operator from $\cdom B$ to
$\cdom A$ with $\|D\| \le c$. It has a unique extension, again
denoted by $D$, from $\cdom B$ to $\cdom A$ with $\|D\| \le c$, such
that
\[
 D (B_{\s})^{\times} \subset (A_{\s})^{\times},
\]
or taking adjoints, using \eqref{Bbadj},
\begin{equation}\label{preli}
 A_{\s}=(A_\s)^{\times \times} \subset (D (B_{\s})^{\times})^\times
 = (B_{\s})^{\times \times}D^\times=B_\s W_{0},
\end{equation}
where $W_0=D^{\times}$ is a bounded linear operator from $\cdom A$
to $\cdom B$, with $\|W_0\| =\|D\| \le c$. Observe that the
inclusion  $\dom B^{*} \subset \dom A^{*}$ implies that
\begin{equation}\label{mula}
 \mul A \subset \mul B.
\end{equation}
Now let $\{f,f'\} \in A$, so that $f'=A_{\s}f+\varphi$ with $\varphi \in \mul A$.
By \eqref{mula} one has  $\varphi \in \mul B$. By \eqref{preli} the inclusion
$\{f, A_{\s}f\} \in A_{\s}$ implies that
\[
 \{f,W_0f\} \in W_0, \quad \{W_0f, A_{\s}f \} \in B_{\s},
\]
and, hence
\[
 \{W_0f, A_{\s}f +\varphi\} \in B.
\]
Therefore one concludes that $\{f,f'\} \in BW_0$, i.e., $A \subset
BW_0$ holds with $W_{0} \in \mathbf{B}(\cdom A, \cdom B)$.
Finally, let $W$ be the zero continuation of $W_0$ to
$(\dom A)^\perp$. Then $W\in \mathbf{B}(\sH_A,\sH_B)$ with
$\|W\|=\|W_0\|$ and, moreover, the inclusion $A\subset BW$ is
satisfied.
\end{proof}

In particular, the equivalences $AA^*\le BB^*$ $\Leftrightarrow$
$W^{*}B^*\subset A^*$ $\Leftrightarrow$ $A\subset BW$ with $\|W\|\le
1$ can be found in \cite[Proposition~2.2, Remark~2.3]{AH}. For
densely defined operators $A$ and $B$ the implication $AA^*\le BB^*$
$\Rightarrow$ $A\subset BW$, $\|W\|\le 1$, can be found in
\cite[Theorem~2]{D}. \\

The following two corollaries are variations on the theme of
Theorem \ref{een}.

\begin{corollary}\label{cornew1+}
Let $A$ and $B$ be closed linear relations as in Theorem~\ref{een}
and, in addition,
let $T\in \mathbf{B}(\sK, \sH_A)$ with $\sK$ a Hilbert space. Then
\[
  A A^{*} \leq c^{2} B B^{*} \quad \Rightarrow \quad
  AT\overline{T^{*}A^{*}} \leq c^{2} \|T\|^2 BB^{*},
\]
where $c\geq 0$.
In particular,
\[
BT\overline{T^{*}B^{*}}\le \|T\|^2BB ^{*}
\]
holds for every $T\in \mathbf{B}(\sK, \sH_A)$.
\end{corollary}

\begin{proof}
Assume that $ AA^{*} \leq c^{2} B B^{*}$, which
by Theorem \ref{een}  is equivalent to  the inclusion $A \subset BW$.
Therefore it follows that
\[
A T \subset BWT.
\]
Observe that $AT$ is closed and that $WT$ is bounded.
Hence again by Theorem \ref{een} one obtains
\[
AT (AT)^{*} \le \|WT\|^{2} BB^{*}
\]
Now observe that \eqref{Bbadj} shows that
\[
 (AT)^{*}= ((T^{*}A^{*})^{*})^{*}= \overline{T^{*}A^{*}}
\]
Hence this leads to
\[
 AT \overline{T^{*}A^{*}} \le c_{T}^{2} B B^{*},
\]
where one can take $c_T=\|WT\|\le c\|T\|$.
The last statement follows from the first one with
the choices $A=B$ and $c=1$.
\end{proof}

\begin{corollary}\label{cornew0}
Let $A$ and $B$ be closed linear relations
as in Theorem~\ref{een} and
let $T$ be a linear relation from the Hilbert space $\sH$
to the Hilbert space $\sK$.
Then
\begin{equation}\label{impl1}
 AA^{*} \le c^{2} BB^{*} \quad \Rightarrow \quad
 \overline{TA}(TA)^* \leq c^{2} \overline{TB}(TB)^*,
\end{equation}
where $c\geq 0$. In particular, if $T\in \mathbf{B}(\sH,\sK)$ then
\[
 AA^{*} \le c^{2} BB^{*} \quad \Rightarrow
 \quad \overline{TA}A^*T^* \leq c^{2} \overline{TB}B^*T^*.
\]
\end{corollary}

\begin{proof}
Assume that  $AA^{*} \le c^{2} BB^{*}$. Then by Theorem \ref{een}
$A\subset BW$ for some $W \in \mathbf{B}(\sH_A,\sH_B)$ with $\|W\| \le c$.
Hence it follows that
\begin{equation}\label{tabw}
 TA \subset T(BW)=(TB)W \subset \overline{TB} W.
\end{equation}
Due to \eqref{Bbadj} the following identity holds
\[
 \overline{TB}W =(W^*(TB)^*)^*,
\]
which implies that the relation $\overline{TB}W$ is closed.
Therefore  one concludes from \eqref{tabw}
that
\[
 AA^{*} \le c^{2} BB^{*} \quad \Rightarrow \quad
 \overline{TA}\subset \overline{TB}W.
\]
By Theorem \ref{een} this implication can be rewritten
as the implication stated in \eqref{impl1}.
If $T\in \mathbf{B}(\sH,\sK)$ the last statement is obtained
by applying \eqref{Bbadj} to \eqref{impl1}.
\end{proof}

The occurrence of the equality $A = BW$ in Theorem \ref{een} can be
characterized as follows.

\begin{proposition}
Let $\sH_A$, $\sH_B$, and $\sH$ be Hilbert spaces, let $A$ be a
closed linear relation from $\sH_A$ to $\sH$, and let $B$ be a
closed linear relation from $\sH_B$ to $\sH$. Then there exists a
bounded {\rm{(}}not necessarily closed\,{\rm{)}} operator $W$ from $\dom A$ into
$\cdom B$ such that
\begin{equation}\label{A=BC}
 A = BW,
\end{equation}
if and only if the following conditions are satisfied
\begin{enumerate}\def\labelenumi{\rm (\roman{enumi})}
\item  the inequality \eqref{AB} holds for some $c \ge 0$;
\item $\mul A=\mul B$.
\end{enumerate}
\end{proposition}

\begin{proof}
($\Rightarrow$) If $A = BW$ holds for some bounded operator $W$ from
$\dom A$ into $\cdom B$, then clearly $A\subset BW^{**}$ and here
$W^{**} \in \mathbf{B}(\cdom A, \cdom B)$, since $\dom A\subset \dom
W$. Now the inequality \eqref{AB} is obtained from Theorem
\ref{een}. Since $W$ is an operator, one obtains $\mul A=\mul
BW=\mul B$; see \eqref{mulAB}.

($\Leftarrow$) The inequality \eqref{AB} implies the existence of
$W_0 \in \mathbf{B}(\cdom A, \cdom B)$ such that $A \subset BW_0$ by
Theorem \ref{een}. Then $\dom A\subset \dom W_0$ and the restriction
$W:=W_0\upharpoonright\dom A$ is a bounded operator such that
$A\subset BW$ and $\dom BW=\dom A$. The second assumption implies
that $\mul BW=\mul B=\mul A$ and hence the equality $A = BW$ follows
from Corollary \ref{arens1}.
\end{proof}

The following result concerns the alternative formulation of the Douglas
lemma which is known in the literature, but now in the context of relations.  
The domain condition is a sufficient condition.

\begin{proposition}\label{NEW}
Let $\sH_A$, $\sH_B$, and $\sH$ be Hilbert spaces, let $A$ be a
closed linear relation from $\sH_A$ to $\sH$, and let $B$ be a
closed linear relation from $\sH_B$ to $\sH$.
Assume that $\dom A^{*}=\dom B^{*}$.
Then the following statements
are equivalent:
\begin{enumerate}\def\labelenumi{\rm (\roman{enumi})}
\item $AA^{*} \le c^{2} BB^{*}$ for some $c \ge 0$;
\item $AA^{*}=BMB^{*}$ for some $0 \le M \in \mathbf{B}(\sH_{B})$
with $\|M\| \le c^{2}$.
\end{enumerate}
\end{proposition}

\begin{proof}
(i) $\Rightarrow$ (ii) By Theorem \ref{een}
it follows that $A \subset BW$, and that
$W^{*}B^{*} \subset A^{*}$. Let $Q$ be the
orthogonal projection onto
$(\mul A^{*})^{\perp}$. Then clearly
\[
 QW^{*}B^{*}  \subset QA^{*},
\]
where $QA^{*}$ is an operator.
The assumption $\dom A^{*}=\dom B^{*}$ implies
that actually equality holds
\[
 QW^{*}B^{*} = QA^{*}.
\]
Therefore one obtains via 
$AA^{*}=AQA^{*}$, see Lemma \ref{T*T}, that
\[
 AA^{*}=AQA^{*} \subset BWQA^{*} 
 =BWQW^{*}B^{*}=(BWQ)(QW^{*}B^{*}),
\]
where the relation  $BWQ$ is closed and
\[
 (QW^{*}B^{*})^{*}=B(QW^{*})^{*}=BWQ.
\]
Hence the term $(BWQ)(QW^{*}B^{*})$ is selfadjoint and equality prevails:
\[
  AA^{*}=BWQW^{*}B^{*}=BMB^{*} \quad \mbox{with} \quad M=WQW^{*}.
\]
Note that $\|M\| \le \|W\|^{2} \le c^{2}$.

(ii) $\Rightarrow$ (i)
Since $M\ge 0$ is bounded one can rewrite (ii) in the form
\begin{equation}\label{tekla}
 AA^*=BMB^{*}= BM^{1/2} M^{1/2}B^{*} 
 \subset  (BM^{1/2}) \overline{M^{1/2}B^{*}}.
\end{equation}
Observe that by \eqref{Bbadj}
\[
 \overline{M^{1/2}B^{*}}=(M^{1/2}B^{*})^{**} =(BM^{1/2})^{*}.
\]
This equality and the fact that  $BM^{1/2}$ is closed together
show that both sides in \eqref{tekla} are selfadjoint; see Lemma \ref{T*T}.
Thus there is actually equality in \eqref{tekla}:
\[
AA^*=(BM^{1/2})\overline{M^{1/2}B^{*}}.
\]
Now Corollary \ref{cornew1+} implies that
\[
 AA^*= BM^{1/2}\overline{M^{1/2}B^{*}}\leq \|M\| BB^*,
\]
so that \eqref{AB} follows with $c^2=\|M\|$.
\end{proof}

\section{Domination of linear relations}

The following notions and terminology are strongly influenced by the
theory of Lebesgue type decompositions of linear relations and
forms, cf. \cite{HSS1}, \cite{HSS2}, \cite{STT2013}.
In fact in these papers the notion of
domination is used for (mostly closable) operators. However
domination can be defined also in the context of linear relations as
follows.

\begin{definition}\label{domi}
Let $\sH_A$, $\sH_B$, and $\sH$ be Hilbert spaces, let $A$ be a
linear relation from $\sH$ to $\sH_A$, and let $B$ be a  linear
relation from $\sH$ to $\sH_B$. Then $B$ dominates $A$ if there
exists an operator $Z \in \mathbf{B}(\sH_B, \sH_A)$ such that
\begin{equation}\label{dom}
 ZB \subset A.
\end{equation}
\end{definition}

Note that the inclusion  $ZB \subset A$ in \eqref{dom} means that
\begin{equation}\label{dom+}
\{ \{f,Zf'\}:\, \{f,f'\} \in B\} \subset A.
\end{equation}
This shows that $\dom B \subset \dom A$ and that $\ker B \subset
\ker A$. Furthermore,
\[
 \mul ZB= Z(\mul B) \subset \mul A.
\]
It follows from the definition that $Z$ takes $\ran B$ into $\ran
A$; the boundedness implies that $Z$ takes $\cran B$ into $\cran A$.
Hence one can assume that $(\ran B)^\perp \subset \ker Z$, in which
case $Z$ is uniquely determined. Domination is transitive: if
$Z_1B\subset A$ and $Z_2C\subset B$ then
\[
 Z_1(Z_2C)\subset Z_1 B\subset A,
\]
so that $(Z_1Z_2)C\subset A$.

Let $A$ and $B$ be relations in a Hilbert space $\sH$
which satisfy $B \subset A$.
Then clearly $B$ dominates $A$ (with $Z=1$). In particular, since
$A \subset A^{**}$, it follows  that $A$ dominates $A^{**}$. \\

In the particular case when $A$ and $B$ in the above definition are
linear operators it is possible to give an equivalent
characterization of domination.

\begin{lemma}\label{domlemma}
Let $\sH_A$, $\sH_B$, and $\sH$ be Hilbert spaces,
let $A$ be a  linear operator from $\sH$ to $\sH_A$,
and let $B$ be a  linear operator from $\sH$ to $\sH_B$.
Then $B$ dominates $A$ if and only if there exists $c \ge 0$ such that
 \begin{equation}\label{dominate}
 \dom B \subset \dom A \quad \text{and}
  \quad \|A f\|\le c\|B f\|,
 \quad f\in \dom B.
\end{equation}
\end{lemma}

\begin{proof}
Assume that $B$ dominates $A$.  Then  \eqref{dom}
shows that $\dom B \subset \dom A$
and that for all $f \in \dom B$ one has $ZBf=Af$, which leads to
 \[
\|A f\| \le \|Z\|\|B f\|, \quad f \in \dom B.
\]
The desired result follows from this with $c=\|Z\|$.

Conversely, assume that \eqref{dominate} holds. Define an operator
$Z_0$ from $\ran B$ to $\ran A$ by $Z_0Bf=Af$, $f\in \dom B$. It
follows from \eqref{dominate} that the operator $Z_0$ is well
defined and bounded with $\|Z_0\|\le c$. Thus $Z_0$ can be continued
to a bounded operator from $\cran B$ to $\cran A$ with the same
norm. Let $Z$ be the extension of $\clos Z_0$ obtained by defining
$Z$ to be $0$ on $(\ran B)^\perp$. Then clearly $Z:\sH_B\to \sH_A$
is bounded and $ZB\subset A$ holds.
\end{proof}

A weaker version of Lemma \ref{domlemma} with densely defined
operators on a Banach space appears in \cite[Theorem~2.8]{F}; see
also \cite{BB,E}.

\begin{lemma}\label{blabla}
Let the relation $B$ dominate the relation  $A$  as in \eqref{dom}, then
\begin{equation}\label{dom1}
 A^{*} \subset B^{*} Z^{*},
\end{equation}
and, consequently
\begin{equation}\label{dom2}
 ZB^{**} \subset A^{**}.
\end{equation}
In other words, $B^{**}$ dominates $A^{**}$
with the same operator $Z$. In particular, if $B$ dominates $A$
then the following inclusions are valid
\[
\dom B\subset \dom A, \quad \ran A^*\subset \ran B^*,
\quad \mbox{and} \quad \dom B^{**}\subset \dom A^{**}.
\]
\end{lemma}

\begin{proof}
It follows from \eqref{dom} and \eqref{Bbadj} that
\[
 A^* \subset (ZB)^*=B^*Z^*.
\]
Now taking adjoints again yields
\[
 Z^{**}B^{**} \subset (B^*Z^*)^* \subset A^{**},
\]
and this proves \eqref{dom2}.
The remaining statement are clear from \eqref{dom1} and \eqref{dom2} .
\end{proof}

So far domination has been defined for linear relations
which are not necessarily closed.
Due to Lemma \ref{blabla} domination of
closed linear relations can be characterized in terms of majorization.

\begin{theorem}\label{een+}
Let $\sH_A$, $\sH_B$, and $\sH$ be Hilbert spaces,
let $A$ be a closed linear relation from $\sH$ to $\sH_A$,
and let $B$ be a  closed linear relation from $\sH$ to $\sH_B$. Then
there exists an operator $Z \in \mathbf{B}(\sH_B, \sH_A)$
such that
\begin{equation}\label{AABC+}
 ZB \subset A
\end{equation}
if and only if there exists $c \ge 0$ such that
\begin{equation}\label{AAB}
 A^{*}A \leq c^{2} B^{*}B.
\end{equation}
One can take $\|Z\| \le c$.
\end{theorem}

\begin{proof}
Since $A$ and $B$ are assumed to be closed the inclusions
\eqref{AABC+} and \eqref{dom1} are equivalent. Hence the result
follows from Theorem \ref{een}.
\end{proof}
 
\begin{proposition}\label{Dgbounded}
Let $\sH_A$, $\sH_B$, and $\sH$ be Hilbert spaces, let $A$ be a
closed linear relation from $\sH$ to $\sH_A$, and let $B$ be a
closed linear relation from $\sH$ to $\sH_B$. Then there exists an
operator $Z \in \mathbf{B}(\sH_B, \sH_A)$ such that
\begin{equation}\label{AequalZB}
  A= ZB
\end{equation}
if and only if the following three conditions are satisfied:
\begin{enumerate}\def\labelenumi{\rm (\roman{enumi})}
\item the inequality \eqref{AAB} holds for some $c \ge 0$;
\item $\dom A=\dom B$;
\item $\dim(\mul A)\le \dim(\mul B)$.
\end{enumerate}
\end{proposition}
\begin{proof}
($\Rightarrow$) Property (i) follows directly from Theorem
\ref{een+}. Since $\dom Z =\sH_B$, the equality \eqref{AequalZB}
implies (ii). Finally, it follows from \eqref{mulAB} and
\eqref{AequalZB} that $\mul A=\mul ZB=Z(\mul B)$, i.e. $Z$ maps
$\mul B$ onto $\mul A$, and hence (iii) holds.

($\Leftarrow$) Decompose $A$ and $B$ via their operator parts:
\[
 A=A_{\rm s} \hoplus A_{\rm mul}, \quad
 B=B_{\rm s} \hoplus B_{\rm mul}.
\]
By Lemma \ref{T*T} the condition \eqref{AAB} is equivalent to
$(A_{\rm s})^{*}A_{\rm s} \le c^{2} (B_{\rm s})^{*}B_{\rm s}$, $c\ge
0$. Now by Theorem \ref{een+} there exists $Z_0\in
\mathbf{B}(\sH_B\ominus \mul B, \sH_A\ominus\mul A)$ such that
\[
 Z_0B_{\rm s}\subset A_{\rm s}.
\]
By the condition (ii) $\dom A_{\rm s}=\dom B_{\rm s}$ and hence, in
fact, the equality $Z_0B_{\rm s} = A_{\rm s}$ prevails. Moreover,
the condition (iii) guarantees the existence of a surjective
operator $Z_m\in \mathbf{B}(\mul B,\mul A)$. Finally, by taking
$Z=Z_0\oplus Z_m$ one gets the desired identity $ZB=A$.
\end{proof}

Finally note that the result in Proposition \ref{NEW} has a counterpart
in the setting of Theorem \ref{een+}.
Let $\sH_A$, $\sH_B$, and $\sH$ be Hilbert spaces, let $A$ be a
closed linear relation from $\sH$ to $\sH_A$, and let $B$ be a
closed linear relation from $\sH$ to $\sH_B$.
If $\dom A=\dom B$, then the following statements are equivalent:
\begin{enumerate}\def\labelenumi{\rm (\roman{enumi})}
\item $A^{*}A \le c^{2} B^{*}B$, \,\, $c\ge 0$;
\item $A^*A= B^*MB$ for some 
$0\leq M \in \mathbf{B}(\sH_B)$ with $\|M\|\leq c^2$.
\end{enumerate}

\section{Majorization and domination} 

There is a direct connection between the majorization 
of bounded operators  as in the original Douglas lemma  
and the notion of domination of linear relations 
as in Definition \ref{domi}. 

\begin{lemma}\label{douglasfact}
Let $\sH_A$, $\sH_B$, and $\sH$ be Hilbert spaces,
let $A\in \mathbf{B}(\sH_A,\sH)$, $B\in \mathbf{B}(\sH_B,\sH)$, 
and $W\in \mathbf{B}(\sH_A,\sH_B)$. Then
\begin{equation}\label{A=BW}
 A=BW  \quad \Leftrightarrow \quad WA^{-1} \subset B^{-1}.
\end{equation}
\end{lemma}

\begin{proof}
First observe that if $H\in\mathbf{B}(\sH_1,\sH_2)$, then
\begin{equation}\label{haha}
 HH^{-1}=I_{\ran H}\subset I_{\sH_{2}},
 \quad  
 H^{-1} H=I_{\sH_{1}}\,{\widehat+}\,\bigl(\{0\}\times\ker H\bigr)\supset I_{\sH_{1}},
\end{equation}
as is clear from \eqref{BBinv} and \eqref{BBinv+}. 

($\Rightarrow$) Assume that  $A=BW$. Then by \eqref{haha} it follows that
\[
 WA^{-1}\subset B^{-1}BWA^{-1} = B^{-1}AA^{-1}\subset B^{-1}.
\]

($\Leftarrow$) Assume that  $WA^{-1} \subset B^{-1}$. 
Then by \eqref{haha} it follows that
\[
  BW \subset BWA^{-1}A \subset BB^{-1}A \subset A,
\]
so that $BW \subset A$.  Actually equality
$BW=A$ prevails here, since both $BW$ and $A$ are 
everywhere defined operators.
\end{proof}

In other words, the lemma expresses the fact that when 
$A$ and $B$ are bounded operators, 
then $B$ majorizes $A$ in the sense of 
$AA^{*} \le \lambda BB^{*}$ (cf. Lemma \ref{Douglas}) 
if and only if the relation $A^{-1}$ dominates the relation $B^{-1}$
in the sense of Definition \ref{domi}

The connection in Lemma \ref{douglasfact} is useful as it yields
a particularly simple proof for the characterization of the ordering 
of nonnegative selfadjoint relations as in \eqref{ineq1}.
For earlier treatments of the ordering, see \cite{CS,HSSW}.

\begin{theorem}\label{twee}
Let $H_1$ and $H_2$ be nonnegative selfadjoint relations in a Hilbert space $\sH$.
Then the following statements are equivalent:
\begin{enumerate}\def\labelenumi{\rm (\roman{enumi})}
\item $H_1 \leq H_2$;
\item $(H_1+x)^{-1} \geq (H_2+x)^{-1}$ for some and hence for every $x>0$;  
\item $H_1^{-1} \geq H_2^{-1}$.
\end{enumerate}
\end{theorem}

\begin{proof}
(i) $\Leftrightarrow$ (ii) 
Recall that $H_{1} \leq H_{2}$ if and only if for some  
(and hence for all) $x >0$
\[ 
H_1+x \leq H_2+x,
\]
and note that for $x>0$  the inverses $(H_1+x)^{-1}$ and $(H_2+x)^{-1}$ 
belong to $\mathbf{B}(\sH)$.
By Theorem \ref{een+} 
$H_1+x \leq H_2+x$ is equivalent to the existence of $Z \in
\mathbf{B}(\sH)$ such that
\begin{equation}\label{incl0}
 Z(H_2+x)^{1/2} \subset (H_1+x)^{1/2},\quad \|Z\|\leq 1;
\end{equation}
cf. Corollary \ref{repr+}. Now an application of Lemma \ref{douglasfact}
shows that \eqref{incl0} is equivalent to 
\begin{equation}\label{incl1}
 (H_2+x)^{-1/2}= (H_1+x)^{-1/2}Z.
\end{equation}
Finally, by Lemma~\ref{Douglas} (or Theorem \ref{een})
\eqref{incl1} is equivalent to  
\[
 (H_2+x)^{-1/2}(H_2+x)^{-1/2}\leq (H_1+x)^{-1/2}(H_1+x)^{-1/2},
\]
since $\|Z\|\leq 1$. 

(ii) $\Leftrightarrow$ (iii) Let $H$ be a nonnegative
selfadjoint relation.  Then clearly also $H^{-1}$ is a nonnegative
selfadjoint relation and it is connected to $H$ via
\begin{equation}\label{eq0}
\left(H+x\right)^{-1}=\frac{1}{x}-\frac{1}{x^2}
\left(H^{-1}+\frac{1}{x}\right)^{-1},
\end{equation}
where $x >0$. Hence for a pair of nonnegative selfadjoint relations
$H_1$ and $H_2$ one obtains for each $x >0$:
\[
 (H_2+x)^{-1} - (H_1+x)^{-1}
 =\frac{1}{x^2} \left[ \left(H_1^{-1}+\frac{1}{x}\right)^{-1}-\left(H_2^{-1}+\frac{1}{x}\right)^{-1} \right].
\]
Now the equivalence is obtained from (i) $\Leftrightarrow$ (ii).
\end{proof}
 
\textbf{Acknowledgement.} The first author is grateful for the
support from the Emil Aaltonen Foundation.

\end{document}